\documentclass[12pt, a4paper]{article}
\usepackage[latin1]{inputenc}
\usepackage{amsmath}
\usepackage{amsfonts}
\usepackage{amssymb}
\usepackage{latexsym}
\usepackage{mathrsfs}
\usepackage{graphicx}
\usepackage{pstricks}
\usepackage{color}
\usepackage{twistor}
\usepackage[all]{xy}
\usepackage[colorlinks=true, linkcolor=black, citecolor=black, urlcolor=black, linktocpage=true]{hyperref}

\numberwithin{equation}{section}

\begin{document}

\title{Moduli stacks of maps for supermanifolds}

\author{\normalsize Tim Adamo \& Michael Groechenig \\ \small \textit{The Mathematical Institute} \\ \small \textit{University of Oxford} \\ \small \textit{24-29 St Giles'} \\ \small \textit{Oxford, OX1 3LB, U.K.}}

\date{}

\maketitle

\abstract{We consider the moduli problem of stable maps from a Riemann surface into a supermanifold; in twistor-string theory, this is the instanton moduli space.  By developing the algebraic geometry of supermanifolds to include a treatment of superstacks we prove that such moduli problems, under suitable conditions, give rise to Deligne-Mumford superstacks (where all of these objects have natural definitions in terms of supergeometry).  We make some observations about the properties of these moduli superstacks, as well as some remarks about their application in physics and their associated Gromov-Witten theory.}

\pagebreak

\tableofcontents

\section{Introduction}

The study of moduli spaces in algebraic geometry is an old and storied topic which has proven to be an exciting field for a wide variety of reasons.  Not least among these is the fact that moduli problems present one of the best examples for the use of stack theory in algebraic geometry (c.f., \cite{stacks-project, Gomez:1999}).  Furthermore, the development of string theory over the past twenty-five years has given further impetus to the study of moduli problems.  In particular, the classic study of instanton counting for the quintic Calabi-Yau in $\P^{4}$ \cite{Candelas:1990rm} led to the prominence of Kontsevich's moduli space of stable maps \cite{Kontsevich:1995} and Gromov-Witten theory as a means to study the interaction between enumerative algebraic geometry and mirror symmetry (c.f., \cite{Cox:1999, Vafa:2003}).  The moduli space of stable maps from a Riemann surface into an abstract variety or scheme is by now relatively well understood; indeed it was shown some time ago that this moduli problem gives rise to a Deligne-Mumford stack as long as the target space is projective and smooth \cite{Behrend:1995}.  In this paper we follow the tradition of physics suggesting directions for mathematics by studying a generalization of this moduli problem to a setting where the target space is a supermanifold.

A supermanifold can be na\"{i}vely thought of as a manifold in the usual sense, but now endowed with some coordinates that anti-commute.  Indeed, the earliest definitions of supermanifolds are as locally-ringed spaces which are locally homeomorphic to the super-space $\R^{n|m}$; a space with $n$ real commuting or bosonic degrees of freedom and $m$ anti-commuting or fermionic degrees of freedom \cite{DeWitt:1992}.  The natural analogue of supermanifolds in the theory of algebraic geometry is the concept of super-scheme.  This algebraic perspective was first developed in \cite{Berezin:1975, Kostant:1977, Leites:1980}, and corresponds to an ordinary (bosonic) scheme endowed with a structure sheaf of super-algebras obeying an obvious consistency condition \cite{Bern-96, Manin:1997}. These constructions have been of interest to physicists for some time due to their obvious applicability to the study of supersymmetric field theories (i.e., theories with a symmetry that intertwines bosonic and fermionic degrees of freedoms) and supergravity (c.f., \cite{Rogers:2007}).  However, recent developments in the study of the planar (i.e., large number of colors) sector of maximally supersymmetric ($\cN=4$) super-Yang-Mills (SYM) theory have demonstrated a novel application of super-geometry and highlighted the need for a rigorous theory of the moduli space of stable maps from a Riemann surface to a super-scheme.

In 2003, Witten discovered that particular classes of $n$-particle scattering amplitudes in planar $\cN=4$ SYM can be calculated by integrating over the moduli space of stable maps of degree $d$ from a Riemann surface of genus $g$ into the super-scheme $\P^{3|4}$ (denoted $\bar{\cM}_{g,n}(\P^{3|4}, d)$), where $d+1-g$ is the number of negative helicity particles involved and $g$ is the loop-order of the corresponding field theory calculation \cite{Witten:2003nn}.  For instance, a Maximal Helicity Violating (MHV) tree interaction involves an arbitrary number of positive helicity particles (gluons for Yang-Mills theory) and two negative helicity particles. This amplitude is supported on a degree 1, genus zero (as tree level indicates zero loops) holomorphic curve in $\P^{3|4}$ (i.e., a line).  The target space $\P^{3|4}$ of this theory has the interpretation as the twistor space of $(4|8)$-dimensional chiral Minkowski super-space, the natural space-time of $\cN=4$ SYM (see \cite{Manin:1997, Wolf:2006me} for good reviews of algebraic super-geometry in the context of twistor theory).  

Although Witten's original formulation is by no means unique, the CFT correlator computations associated to scattering amplitudes involve (in one way or another) integrals over the moduli space $\bar{\cM}_{g,n}(\P^{3|4},d)$ in all known twistor-string theories \cite{Berkovits:2004hg, Berkovits:2004tx, Mason:2007zv}.  In much of the literature up to this point, it has been assumed that the properties of $\bar{\cM}_{g,n}(\P^{3|4},d)$ are inherited from its well-studied bosonic counterpart; these include compactness, smoothness (at genus zero), algebraicity, and the Deligne-Mumford property at the level of the stack.  Beyond twistor-string theory itself, this moduli space and its properties have played a role in a myriad of related advances using twistor methods (see \cite{Adamo:2011pv} for a review).  These include: the relationship between the connected and disconnected prescriptions for the twistor-string \cite{Gukov:2004ei}; the embedding of the Grassmannian formalism of \cite{ArkaniHamed:2009si} into the twistor-string \cite{Bullimore:2009cb}; the proof for the BCFW recursion relations \cite{Britto:2004ap, Britto:2005fq} in twistor-string theory \cite{Skinner:2010cz}; the proof of the BCFW recursion relations for the supersymmetric Wilson loop in twistor space \cite{Mason:2010yk, Bullimore:2011ni}; and the derivation of all-loop recursion relations for mixed Wilson loop-local operator correlators \cite{Adamo:2011cd}.  Despite the central role played by $\bar{\cM}_{g,n}(\P^{3|4},d)$ in each of these examples, there has never been a rigorous investigation of this moduli space to verify whether or not its supposed properties actually hold.

In this paper, we provide a rigorous treatment of the moduli stack of stable maps from a Riemann surface to a general super-scheme.  Our main result is the following theorem:
\begin{quotation}
\noindent\textbf{Theorem}  \textit{Let $X$ be a smooth, projective and split super-scheme, and let $\bar{\cM}_{g,n}(X,\beta)$ be the moduli stack of stable maps from a Riemann surface of genus $g$ with $n$ marked points to $X$ whose image lies in the homology class $\beta$ (see Definition \ref{Stablemap}).  Then $\bar{\cM}_{g,n}(X,\beta)$ is a Deligne-Mumford super-stack.}
\end{quotation}
In the case of twistor-string theory, this confirms the working assumptions of the literature discussed earlier.

Section \ref{SEGA} builds a theory of super-schemes which is largely a review or reformulation of studies presented elsewhere in the literature.  We then take a new step by introducing the notion of a super-stack, as well as the corresponding concepts of algebraicity and Deligne-Mumford-ness for these objects.  Our treatment is heavily influenced by Behrend's presentation of stacks at the Isaac Newton School on Moduli Spaces \cite{Behrend:2011}, as well as the classic reference on algebraic stacks \cite{Laumon:2000}.  We then use our theory of super-stacks to show that the stack of stable maps from a Riemann surface to a super-scheme is a Deligne-Mumford super-stack in \S \ref{Moduli}, under natural assumptions.  We also make some observations about other properties of this moduli space which have been used in the physics literature, including its natural maps and smoothness criteria.  In Section \ref{Conclusion}, we make some observations about open issues and applications of our super-stack construction to theoretical physics, as well as laying out some questions about the Gromov-Witten theory associated with such objects.

%%%%%%%%%%%%%%%%%%%%%%%%%%%%%%%%%%%%%%%%%%%%%%%%%%%%%%%%%%%%%%%%%%%%%%%%%%%%%%%%%%%%%%%%%%%%%%

\subsection{A Note on Terminology}

As our primary interest throughout this paper will be ``super-geometric'' objects (e.g., super-manifolds, super-schemes, etc.), we must make a choice about how to differentiate these constructs from their counterparts which are familiar from ordinary algebraic or differential geometry.  To avoid a proliferation of adjectives, we simply drop the ``super-'' from in front of all super-geometric objects we consider after their initial introduction and definition.  Hence, from now on we will use ``scheme'' to refer to a super-scheme, ``stack'' to refer to a super-stack, and so forth (although we will emphasize the distinction whenever a new ``super-'' construct is introduced).  The objects and constructions of ordinary algebraic geometry will be distinguished by placing a ``bosonic'' before them; thus a classical scheme will be refered to as a ``bosonic scheme'' for the remainder of the paper. All constructions will take place over the field of complex numbers and we will restrict our attention to schemes which are locally of finite type.

%%%%%%%%%%%%%%%%%%%%%%%%%%%%%%%%%%%%%%%%%%%%%%%%%%%%%%%%%%%%%%%%%%%%%%%%%%%%%%%%%%%%%%%%%%%%%%
%%%%%%%%%%%%%%%%%%%%%%%%%%%%%%%%%%%%%%%%%%%%%%%%%%%%%%%%%%%%%%%%%%%%%%%%%%%%%%%%%%%%%%%%%%%%%%  

\section{Algebraic Super-Geometry}
\label{SEGA}

We now turn to the development of the concepts in ``algebraic super-geometry'' which will be necessary to construct the moduli spaces of interest.  From now on, we consider target spaces using the language of schemes rather than manifolds; in the context of twistor-string theory, this is natural since projective varieties are naturally projective schemes.  We first develop the notion of a super-scheme along the lines of prior research in super-geometry (e.g., \cite{Bern-96, Manin:1997}) and study the local properties of morphisms between these objects.  The remainder of the section is devoted to the development of a suitable theory of super-stacks.  A super-geometric generalization of the GAGA principle \cite{Serre:1956, Topiwala:1991} means that we are free to interpolate between the abstract algebraic point of view, which we take through most of this paper, and the more intuitive complex-analytic approach.  Unless otherwise mentioned, we work over the field $\C$.

%%%%%%%%%%%%%%%%%%%%%%%%%%%%%%%%%%%%%%%%%%%%%%%%%%%%%%%%%%%%%%%%%%%%%%%%%%%%%%%%%%%%%%%%%%%%%%
   
\subsection{Super-schemes}

\subsubsection{\textit{Super-Rings}}

In order to build algebraic geometry for supersymmetric settings, we must first consider more basic algebraic notions such as rings.  We define a super-ring $R$ to be a ring in the usual sense, but now with fermionic part:
\begin{equation*}
R=R_{b}\oplus R_{f},
\end{equation*}
where $R_{b}$ is an ordinary (i.e., bosonic) ring and $R_{f}$ is made up of anti-commuting elements.  In other words, a super-ring is a $\Z_2$-graded ring which is graded commutative.  A super-ideal $I\subset R$ is an ideal of $R$ in the usual sense and obeys:
\be{eqn: ideal}
I= \left(I\cap R_{b}\right)\oplus\left(I\cap R_{f}\right).
\ee
We say that $I$ is prime if for all $a,b\in R$, $ab\in I$ implies that $a$ or $b$ is in $I$, as usual.  We then define the ``super-spectrum'' (henceforth, ``spectrum'') of $R$, $\mathrm{Spec}(R)$ to be the set of all prime super-ideals in $R$.  As any ring is trivially a super-ring, there is no need for us to distinguish between the notions of spectrum and super-spectrum.  This leads to the following observation relating the spectrum of a super-ring to the spectrum of its bosonic part:

\begin{lemma}\label{ideals} 
Let $R=R_{b}\oplus R_{f}$ be any super-ring.  Then $\mathrm{Spec}(R)\cong\mathrm{Spec}(R_{b})$ as sets, and under this isomorphism every prime super-ideal $\mathfrak{p}\in\mathrm{Spec}(R)$ gets mapped to its bosonic part.
\end{lemma}
\proof  Let $\mathfrak{p}$ be a prime super-ideal of the super-ring $R$.  Then we can write
\begin{equation*}
\mathfrak{p}=(\mathfrak{p}\cap R_{b})\oplus (\mathfrak{p}\cap R_{f})\equiv \mathfrak{p}_{b}\oplus\mathfrak{p}_{f}
\end{equation*}
by the definition of a super-ideal \eqref{eqn: ideal}.  Now, for any fermionic element $\psi\in R_{f}$, it follows that $\psi^{2}=0$, which indicates that $\psi^{2}\in\mathfrak{p}$.  But $\mathfrak{p}$ is prime, so $\psi\in\mathfrak{p}$ for all $\psi\in R_{f}$.  Hence, we have
\begin{equation*}
\mathfrak{p}=(\mathfrak{p}\cap R_{b})\oplus R_{f}.
\end{equation*}
So $\mathfrak{p}_{f}=R_{f}$, and a prime super-ideal in $\mathrm{Spec}(R)$ carries no new information in its fermionic sector.  Thus, we get the isomorphism $\mathrm{Spec}(R)\cong\mathrm{Spec}(R_{b})$ by projecting out the trivial $R_{f}$ in the direct sum:
\begin{equation*}
\mathfrak{p}_{b}\oplus R_{f}\mapsto\mathfrak{p}_{b},
\end{equation*}
as required.     $\Box$

We now see that the natural definition for a locally super-ringed space is a (bosonic) topological space $X$ equipped with a structure sheaf of super-algebras $\cO_{X}$ such that the stalk of $\cO_{X}$ over each point in $X$ has the structure of a super-ring.  Although locally ringed spaces are the road to bosonic schemes in ordinary algebraic geometry (c.f., \cite{EGA:2, Hartshorne:1977}), Lemma \ref{ideals} allows us to take a much simpler approach to defining super-schemes that builds directly on the already existing theory of bosonic schemes.

\subsubsection{\textit{Super-schemes}}

We now give the definition of super-scheme that will be used in the remainder of this paper:
\begin{defn}[Super-scheme]\label{sschemes}
A \emph{super-scheme} $X$ (henceforth, a \emph{scheme}) is a pair $(X_{b}, \mathfrak{A})$, where $X_{b}$ is a bosonic $\C$-scheme in the ordinary sense and $\mathfrak{A}$ is a quasi-coherent sheaf\footnote{In the context of supersymmetry, this can be restricted to coherent sheaves, as the supersymmetry algebra is finitely generated} of superalgebras, whose bosonic part obeys $\mathfrak{A}_{b}=\cO_{X_{b}}$.
\end{defn}
So a scheme is just a bosonic scheme augmented by a structure sheaf of super-algebras; this definition coincides with that often given in the literature of super-geometry (e.g., \cite{Manin:1997}).  

We now provide some additional definitions which fill out the theory of schemes:
\begin{defn}[Morphism]\label{morph}
A \emph{morphism} of schemes $(X_{b},\mathfrak{A})\rightarrow(Y_{b},\mathfrak{B})$ is a pair $(f,\phi)$, where $f: X_{b}\rightarrow Y_{b}$ is a morphism of bosonic $\C$-schemes and $\phi: \mathfrak{B}\rightarrow f_{*}\mathfrak{A}$ is a morphism of super-algebras.
\end{defn}
\begin{defn}[Sub-scheme]\label{subsch}
A \emph{sub-scheme} $X$ of a scheme $Y=(Y_{b},\mathfrak{B})$ is a scheme $(X_{b},\mathfrak{A})$ equipped with a pair $(f,\phi)$ such that $f: X_{b}\hookrightarrow Y_{b}$ is a closed immersion and $\phi: \mathfrak{B}\rightarrow f_{*}\mathfrak{A}$ is surjective.
\end{defn}

These notions now provide us with a well-defined category of schemes, which we denote as $\mathsf{Sch}$ (the category of bosonic schemes is denoted $\mathrm{B}\mathsf{Sch}$).  Now, any $X\in\ob(\mathsf{Sch})$ comes equipped with a natural forgetful functor $\mathsf{Sch}\rightarrow\mathrm{B}\mathsf{Sch}$ which simply projects to the underlying bosonic scheme:
\begin{equation*}
X=(X_{b},\mathfrak{A})\mapsto X_{b},
\end{equation*}
but there is another important way in which bosonic schemes may be obtained from schemes.  We consider first the case of an affine scheme (i.e., one which is of the form $\mathrm{Spec}(R)$ for some super-ring $R$), and trust that the reader is capable of extending the arguments to a general scheme.  

Note that in any super-ring $R=R_{b}\oplus R_{f}$, the bosonic portion $R_{b}$ must contain nilpotent elements of the form $\psi\chi$, where $\psi, \chi\in R_{f}$.  In general such elements are bosonic and non-zero in $R_{b}$, but must of course square to zero.  In all physical applications, sections of this sort do not appear in the bosonic portion of the effective field theory; that is, after integrating out fermionic degrees of freedom in any Lagrangian, we are left with objects that are ``purely bosonic'' in the sense that they are not inherited from the fermions in the manner just described.  We therefore consider a construction which identifies elements in a super-ring obtained in this fashion with zero:

\begin{defn}[Bosonic truncation]\label{bfun}
Let $R=R_{b}\oplus R_{f}$ be a super-ring.  The ring $\tau_b(R)$ is defined as $R/(R\cdot R_{f})$.  For a scheme $X=(X_{b},\mathfrak{A})$, this allows us to define $\tau_b(X)$, as schemes are glued from affine schemes.\footnote{An equivalent construction for supermanifolds is given in \cite{Manin:1997}.}
\end{defn}
Given $X\in\ob(\mathsf{Sch})$ and $Y\in\ob(\mathrm{B}\mathsf{Sch})$, we say that $X$ is a \textit{superization} of $Y$ if $\tau_b(X)\cong Y$. 

We now establish several important facts about $\tau_b$, beginning with its universal property:

\begin{lemma}\label{univ}
Let $Y\rightarrow X$ be a morphism of schemes, where $Y=Y_{b}$ is a bosonic scheme.  Then there exists a unique morphism $Y\rightarrow\tau_b(X)$ such that the composition $Y\rightarrow\tau_b(X)\rightarrow X$ agrees with the given morphism $Y\rightarrow X$.  In particular,
\begin{equation*}
\Hom(Y,X)=\Hom(Y,\tau_b(X)).
\end{equation*}
\end{lemma}
\proof The lemma is proved for all schemes by proving it for affine schemes.  In this case, we need only show that for a morphism of super-rings $f:A\rightarrow B$, where $B$ is bosonic, we obtain a unique factorization $A\rightarrow\tau_b(A)\rightarrow B$.  As $f$ must map all fermionic elements of $A$ to zero, all elements of $A\cdot A_{f}$ must be mapped to zero.  Hence, we obtain the required factorization map:
\begin{equation*}
\tau_b(A)=A/(A\cdot A_{f})\rightarrow B,
\end{equation*}
which completes the proof.     $\Box$

We know that $\tau_b$ maps schemes to bosonic schemes, but how does it act on morphisms of schemes?  The following result confirms that $\tau_b$ indeed acts as a functor.
\begin{lemma}\label{func}
The map $\tau_b$ can be extended to a functor $\tau_b:\mathsf{Sch}\rightarrow\mathrm{B}\mathsf{Sch}$.
\end{lemma}
\proof  We must check that for a given morphism of schemes $f:Y\rightarrow X$, there is a morphism of bosonic schemes $\tau_b(f):\tau_b(Y)\rightarrow\tau_b(X)$ which preserves identities and respects composition.  This is immediate from Lemma \ref{univ} though, as $Y\rightarrow X\rightarrow\tau_b(X)$ factorizes into $\tau_b(Y)\rightarrow\tau_b(X)$, and the universal property implies that identity and composition are preserved.     $\Box$

Finally, the universal property and the locality of fiber product confirm immediately that the functor $\tau_b$ respects the fiber product: 
\begin{lemma}\label{fp}
Let $Y\rightarrow X$ and $Z\rightarrow X$ be morphisms of schemes.  Then there is a natural isomorphism
\begin{equation*}
\tau_b(Y)\times_{\tau_b(X)}\tau_b(Z)\rightarrow \tau_b\left(Y\times_{X}Z\right).
\end{equation*}
\end{lemma}

Every bosonic scheme gives rise trivially to a scheme.  In the spirit of Grothendieck's relativization of absolute notions, we look to extend the notion of bosonic scheme to a property of morphisms between schemes:
\begin{defn}[Bosonic morphism]\label{bosonic}
A morphism of schemes $Y\rightarrow X$ is called \emph{bosonic} if for every bosonic scheme $U_{b}$ together with a morphism $U_{b}\rightarrow X$, the base-change $Y\times_{X}U_{b}$ is a bosonic scheme.
\end{defn}

We then obtain an important result linking bosonic morphisms of schemes with the $\tau_b$ functor:
\begin{lemma}\label{bpi}
Let $Y\rightarrow X$ be a bosonic morphism of schemes.  Then
\begin{equation*}
\tau_b(X)\times_{X}Y\cong\tau_b(Y).
\end{equation*}
\end{lemma}
\proof  By Definition \ref{bosonic}, it follows that $\tau_b(X)\times_{X}Y$ is bosonic by assumption.  From the natural morphism $\tau_b(X)\times_{X}Y\rightarrow Y$ and Lemma \ref{univ}, it follows that there is a canonical morphism $\tau_b(X)\times_{X}Y\rightarrow\tau_b(Y)$, and by the universal property of fiber products, there is also a natural morphism $\tau_b(Y)\rightarrow\tau_b(X)\times_{X}Y$.  Due to the universal properties that both morphisms satisfy, it follows that they are mutually inverse, and hence $\tau_b(X)\times_{X}Y\cong\tau_b(Y)$.     $\Box$

Finally, we provide a superized notion of finite presentation for the structure sheaf of super-algebras of a scheme:
\begin{defn}[Fermionically of Finite Presentation]\label{ffp}
A scheme $X=(X_{b},\mathfrak{A})$ is called \emph{fermionically of finite presentation} if $\mathfrak{A}_{f}$ is coherent as a $\mathfrak{A}_{b}$-module.
\end{defn}

We often use these properties of the $\tau_b$ functor to define concepts for schemes in terms of the underlying concept for bosonic schemes.  For instance, we say that a scheme $X$ is \emph{projective} if it is fermionically of finite presentation and $\tau_b(X)$ is a projective bosonic scheme.  The category of projective schemes will be denoted $\P\mathsf{Sch}$. 

In the following definition we record the observation that we can take the relative spectrum of a sheaf of super-algebras on a bosonic scheme to obtain a super-scheme. It is constructed by taking $\mathrm{Spec}$ of a super-algebra Zariski locally and gluing the schemes together.

\begin{defn}[Relative spectrum, Split scheme]
Let $X$ be a bosonic scheme and $\cA$ a sheaf of super-algebras on $X_b$. Then there exists a scheme $\mathrm{Spec}_{X}(\cA)$ over $X$ called the \emph{relative spectrum} of $\cA$. A super-scheme $Y$ is called \emph{split}, if there exists a coherent locally free sheaf $\mathcal{V}$ on $\tau_b(Y)$ such that $Y \cong \mathrm{Spec}_{\tau_{b}(Y)}(\wedge^{\bullet} \mathcal{V})$.
\end{defn}

We can apply this definition to an easy example with important applications in twistor-string theory:

\subsubsection*{\textit{Example:} $\P^{m|n}$}

We define $\P^{m|n}$ and show that it is a superization of $\P^{m}$. As an analytic super-manifold, we can chart $\P^{m|n}$ with local coordinates
\begin{equation*}
(Z^{1},\ldots, Z^{m+1}, \psi^{1},\ldots, \psi^{n}),
\end{equation*}
where the $Z^{\alpha}$ are the ordinary homogeneous coordinates on $\P^{m}$ and the $\psi^{i}$ are anti-commuting fermionic coordinates.  We consider a sheaf of super-algebras on $\P^m$
\be{eqn: ss}
\mathfrak{A}_{\P^{m|n}}=\left(\bigoplus_{k=0}^{n}\wedge^{k}\cO_{\P^{m}}(-1)^{\oplus n}\right).
\ee
Clearly $\tau_{b}(\mathfrak{A}_{\P^{m|n}})=\cO_{\P^{m}}$, so define $\P^{m|n}$ to be the relative spectrum of $\mathfrak{A}_{\P^{m|n}}$:
\begin{equation*}
\P^{m|n}\cong \mathrm{Spec}_{\P^{m}}(\mathfrak{A}_{\P^{m|n}}).
\end{equation*}
Note that although $\tau_b(\P^{m|n}) =  \P^m$, the underlying bosonic scheme of $\P^{m|n}$ is different from $\P^m$, since locally its algebra of functions contains nilpotent elements.  Furthermore, $\P^{m|n}$ endowed with the structure sheaf of super-algebras \eqref{eqn: ss} is trivially a projective scheme, and is fermionically of finite presentation.

\subsubsection{\textit{Local Properties of Morphisms}}

In the study of bosonic algebraic geometry, properties of morphisms such as \emph{smooth} or \emph{\'{e}tale} are essential for learning about relationships between objects in that theory.  While formal definitions of these properties can be found in \cite{Hartshorne:1977}, there is an easy analogy in the category of smooth schemes: a smooth morphism corresponds to a submersion, while an \'{e}tale morphism corresponds to a local diffeomorphism.  We now extend these notions from the world of bosonic algebraic geometry to algebraic super-geometry, assuming that all schemes and morphisms are locally of finite type over $\C$.

\begin{defn}[Smooth scheme]
A scheme is called \emph{smooth} if the morphism $\tau_b(X) \to X$ admits a left-inverse $g: X \to \tau_b(X)$ Zariski-locally on $X$, and we can express $X$ with respect to $g$ as $\mathrm{Spec}_{\tau_b(X)}(\wedge^{\bullet} \mathcal{V}),$ where $\mathcal{V}$ is a locally free sheaf (defined on a Zariski open subset $U \subset \tau_b(X)$).
\end{defn}

We would like to say that a morphism $Y \to X$ is smooth if all fibres are smooth schemes. In order for this to be a sensible definition, we need to introduce the notion of flat morphisms.

If $R$ is a super-ring, then an $R$-module $M$ is said to be flat if the functor $-\otimes_{R}M$ is exact (i.e., sends short exact sequences to short exact sequences).  A morphism of rings $A\rightarrow B$ is then called flat if $B$ is flat as an $A$-module.
\begin{defn}[Flat morphism]\label{flatmorph}
Let $X=(X_{b},\mathfrak{A})$ and $Y=(Y_{b},\mathfrak{B})$ be schemes and $f:Y\rightarrow X$ be a morphism.  Such a morphism of schemes is called \emph{flat} if, for every $y\in Y$, the induced morphism of local rings $\mathfrak{A}_{f(y)}\rightarrow \mathfrak{B}_{y}$ is flat.
\end{defn}

From this definition, we state the following two facts about flat morphisms:
\begin{lemma}\label{flat1}
Let $f:Y\rightarrow X$ be a closed immersion corresponding to the ideal sheaf $\mathcal{I}_{X}$, and $g: Z\rightarrow X$ a flat morphism.  Then the pullback $Y\times_{X}Z$ is the closed immersion corresponding to $\mathcal{I}_{Y}=f^{*}\mathcal{I}_{X}$.
\end{lemma}
\proof  Without loss of generality, we consider the affine case, where $f$ corresponds to a surjective morphism of rings $A\rightarrow B$ and $g$ is a flat morphism of rings $A\rightarrow C$.  Then the base change at the level of schemes induces a surjective morphism of rings $C\rightarrow B\otimes_{A} C$, whose kernel is precisely $I\otimes_{A} C$ as $A\rightarrow C$ is flat.  The tensoring of modules corresponds to pullback, and we have the desired result.     $\Box$

\begin{lemma}\label{flat2}
Let $g:Y\rightarrow X$ be a bosonic and flat morphism of schemes, with $X=(X_{b},\mathfrak{A})$, $Y=(Y_{b},\mathfrak{B})$.  Then the fermionic portion of $\mathfrak{B}$ is given by $\mathfrak{B}_{f}=g^*\mathfrak{A}_{f}$, and moreover $\mathfrak{B}^{2}_{f}=g^*\mathfrak{A}^{2}_{f}$. In particular we have that $Y \cong X \times_{X_b} Y_b$.
\end{lemma}
\proof  As $g$ is bosonic, we know from Lemma \ref{bpi} that $\tau_b(Y)= Y\times_{X}\tau_b(X)$, so $\mathfrak{B}/\mathcal{I}_{Y}\cong \mathfrak{B}\otimes_{\mathfrak{A}}\mathfrak{A}/\mathcal{I}_{X}$.  Now, the natural morphism $\tau_b(X)\rightarrow X$ is a closed immersion corresponding to the sheaf of ideals $\mathcal{I}_{X}=\mathfrak{A}\cdot\mathfrak{A}_{f}$, and by Lemma \ref{flat1}, we know that $\mathcal{I}_{Y}=g^{*}\mathcal{I}_{X}$.  Then taking the fermionic portion gives $\mathfrak{B}_{f}=g^*\mathfrak{A}_{f}$.  The second and third statements follow easily from the first.     $\Box$  

These two lemmas allow us to prove the following proposition which will be crucial in our later study of the moduli stack of stable maps from a Riemann surface to a scheme.
\begin{propn}\label{flat3}
Suppose $X=(X_{b}, \mathfrak{A})$ is a scheme which is fermionically of finite presentation, $Y=(Y_{b},\mathfrak{B})$ a scheme, and the morphism $f:Y\rightarrow X$ a flat bosonic morphism.  Then the underlying morphism of bosonic schemes $f_{b}:Y_{b}\rightarrow X_{b}$ is flat.
\end{propn}
\proof Again without loss of generality we can assume that all schemes are affine, so $f$ corresponds to a morphism of rings $A\rightarrow B$.  Using the fact that $f$ is bosonic, it follows that $\tau_b(Y)\rightarrow\tau_b(X)$ is flat as the base-change of the flat morphism $f$.  Now consider a closed immersion $\tau_b(X)\rightarrow X_{b}$ given by the sheaf of ideals $A_{f}^{2}\subset A_{b}$, and likewise for $B$.  By Lemma \ref{flat2}, we know that $B^{2}_{f}=B_{b}\cdot A_{f}^{2}$.  Now, let $I\equiv A\cdot A_{f}$; by assumption $X$ is fermionically of finite presentation, so it follows that $A_{f}^{2}$ is a nilpotent ideal.  We are then in a situation to apply the local criterion for flatness of the $A_{b}$-algebra $B_{b}$, as stated in \cite{Altman:1970}.  As $B_{b}/(I_{b}\cdot B_{b})$ is flat as an $A_{b}/I_{b}$-algebra, it suffices to show that $\Tor^{A_{b}}_{1}(B_{b},A_{b}/I_{b})=0$.

Consider the short exact sequence
\begin{equation*}
0\rightarrow I_{b}\rightarrow A_{b}\rightarrow A_{b}/I_{b}\rightarrow 0
\end{equation*}
Applying the functor $\Tor^{A_{b}}_{*}(B_{b},-)$ to this yields a long exact sequence, of which we are interested in the following portion:
\begin{equation*}
\xymatrix{
\Tor^{A_{b}}_{1}(B_{b},A_{b}) \ar[r] & \Tor^{A_{b}}_{1}(B_{b},A_{b}/I_{b}) \ar[r] & B_{b}\otimes_{A_{b}}I_{b} \ar[r] \ar[d]^{\alpha} & B_{b}\otimes_{A_{b}}A_{b} \ar[d]^{\gamma} \\
& & B\otimes_{A}I \ar[r]^{\beta} & B\otimes_{A}A }
\end{equation*}
Now, as $A_b$ is a free $A_b$-module so $\Tor^{A_{b}}_{1}(B_{b},A_{b})=0$.  Hence, $\Tor^{A_{b}}_{1}(B_{b},A_{b}/I_{b})$ will vanish provided the map $B_{b}\otimes_{A_{b}}I_{b}\rightarrow B_{b}\otimes_{A_{b}}A_{b}$ is injective.  The map $\beta$ is injective as $B$ is a flat $A$-module. Since 
$$ B \otimes_A I = (B \otimes_{A_b} I)/(a_f b \otimes i - b \otimes a_f i : a_f \in A_f,\; b \in B,\; i \in I), $$ we see that $\alpha$ is injective as well. Consequently, $B_{b}\otimes_{A_{b}}I_{b}\rightarrow B_{b}\otimes_{A_{b}}A_{b}$ is injective and $\Tor^{A_{b}}_{1}(B_{b},A_{b}/I_{b})=0$ as desired.     $\Box$  

\begin{defn}[Smooth/\'etale morphism]
A morphism of schemes $f: Y \to X$ is called \emph{smooth} if it is flat and for every point $\mathrm{Spec}(\C) \to X$ the fibre $Y \times_X \mathrm{Spec}(\C)$ is a smooth super-scheme. It is called \emph{\'etale} if it is smooth, bosonic and has zero-dimensional fibres.
\end{defn}

\subsection{Super-stacks}

We now look to build upon the theory of schemes (much of which has built on prior research in super-geometry) and develop a theory of (super-)stacks.  For the reader unfamiliar with this nomenclature, smooth Deligne-Mumford stacks can be thought of as the algebraic generalization of orbifolds in the same way that smooth schemes can be viewed as the algebraic analogue of manifolds.  However, where a coarse moduli space only ``remembers'' a group action at singular points in the original manifold, a stack encodes much more information \cite{EGA:4}.  For the case we will be interested in, the stack will encode the automorphism group of curves in a target space which are the image of a world-sheet of a particular genus under a map of a given degree; an orbifold point in the corresponding coarse moduli space can be thought of as a point where this automorphism group is non-trivial.  A good introduction to stacks is given by \cite{Gomez:1999}. 

Stacks are a generalization of sheaves. The definition in the literature (e.g., \cite{stacks-project, Gomez:1999, Vistoli:2005}) applies directly to the context of supergeometry. Our treatment of algebraic super-stacks on the other hand, should be seen as a special case of relative algebraic geometry \cite{Hakim}. In this theory one replaces the category of rings by a more general category of monoid objects in suitable monoidal category. This approach to algebraic geometry is fundamental to the subject of derived algebraic geometry (see \cite{Toen:2001} for a discussion of this circle of ideas).

\subsubsection{\textit{Schemes as Representable Functors}}

One reason why bosonic schemes can be hard to grasp intuitively is due to the fact that a point of a scheme does not necessarily correspond to a spatial point in the classical sense.  For instance, consider the affine scheme $\A^{1}$ corresponding to the ring of polynomials $\C[t]$.  We know that points of $\A^{1}$ are given by prime ideals in $\C[t]$, so we have a point for every complex number $z\in\C$, namely the prime ideal $(t-z)$.  But we also have the zero ideal $(0)$, which does not correspond to a complex number and therefore to a point of the variety $\A^{1}$.  This is an example of what is called a \textit{generic point} \cite{Hartshorne:1977}.  Geometrically, generic points correspond to irreducible subvarieties; in our example of $\A^{1}$, there is only one non-trivial closed irreducible subvariety: $\A^{1}$ itself.

A way to single out those points of a bosonic $\C$-scheme $X$ which we believe to be points of some variety is by considering the set of morphisms $\{\mathrm{Spec}(\C)\rightarrow X\}$.  Geometrically, $\mathrm{Spec}(\C)$ is a point with the structure sheaf corresponding to constant $\C$-valued functions.  It is therefore not surprising that this set of morphisms, denoted $X(\mathrm{Spec}(\C))$, agrees with what we understand to be the ``set of points'' of a bosonic scheme $X$.  The same concept holds for a manifold $M$ in differential geometry: the set of morphisms $\{\mathrm{pt}\rightarrow M\}$ agrees with the set of points of $M$.

In general, we are free to consider the set of morphisms $Y\rightarrow X$ (for any bosonic scheme $Y$), which we denote by $X(Y)$.  This gives a functor, referred to as the ``functor of points of $X$.''\footnote{Actually, this is a co-functor from the category of bosonic schemes to the category of sets: $(\mathrm{B}\mathsf{Sch})^{\mathrm{op}}\rightarrow\mathsf{Set}$.}  It is a general principle (true in every category) that this functor determines the object uniquely (the Yoneda lemma).  Grothendieck's viewpoint on algebraic geometry is to study a bosonic scheme in terms of its functor of points \cite{EGA:4}.

For example, we can see that $\A^{1}(X)$ is the set of regular functions on $X$ (i.e., $\cO_{X}(X)$).  For $\P^{n}(X)$, a $X$-point of $\P^{n}$ corresponds to a line bundle $L$ together with $n+1$ global sections $\{s_{0},\ldots,s_{n}\}$ which span $L$.  Similar descriptions exist for Grassmanians and flag schemes.

The functorial viewpoint of super-geometry was introduced by the Bernstein school (c.f., \cite{Bern-96}) and emphasized by Manin \cite{Manin:1997}.  Our goal is now to extend these ideas, allowing us to study algebraic spaces and algebraic stacks.  We first recall some necessary notions from category theory:
\begin{defn}[Grothendieck Pretopology]\label{GP}
Let $\mathsf{C}$ be a category.  A \emph{Grothendieck pretopology} on $\mathsf{C}$ is a family of coverings (i.e., a distinguished collection of morphisms $\{U_{i}\rightarrow X\}_{i\in I}$) satisfying the following axioms.
\begin{enumerate}
\item Every isomorphism $Y\rightarrow X$ is a covering.

\item Given a covering $\{U_{i}\rightarrow X\}_{i\in I}$ and a morphism $Y\rightarrow X$, then $\{U_{i}\times_{X}Y\rightarrow Y\}_{i\in I}$ is a covering of $Y$, provided the fiber products exist in $\mathsf{C}$.

\item Given a covering $\{U_{i}\rightarrow X\}_{i\in I}$, and for every $i\in I$ a covering $\{V_{ij}\rightarrow U_{i}\}_{j\in I_{j}}$ of $U_{i}$, then $\{V_{ij}\rightarrow X\}_{(i,j)\in\coprod_{i\in I}I_{j}}$ is a covering of $X$.
\end{enumerate}
\end{defn}

Na\"{i}vely, one can take $\mathsf{C}=\mathrm{B}\mathsf{Sch}$ and consider the Grothendieck pretopology given by the actual open coverings; such a construction generalizes obviously to the category of schemes.  In practice this is not enough to give interesting theories of principal bundles or cohomology with constant coefficients due to the fact that Zariski-open subsets are too big.  This situation is remedied by studying more general ``open subsets,'' not necessarily given by the inclusion of Zariski open subsets.  In other words, we need a suitable notion of locality in the context of category theory which yields non-trivial cohomology.

To do this, we introduce the notion of a set-valued functor as sheaf, as well as the so-called ``fppf'' pretopology:
\begin{defn}\label{funcsheaf}
A set-valued functor $\mathcal{F}:(\mathsf{C})^{\mathrm{op}}\rightarrow\mathsf{Set}$ is a sheaf if, for every covering $\{U_{i}\rightarrow X\}_{i\in I}$ of a $X\in\ob(\mathsf{C})$ and every family of local sections $\{s_{i}\in\mathcal{F}(U_{i})\}_{i\in I}$, there exists a unique section $s\in\mathcal{F}(X)$ satisfying: $\mathcal{F}(U_{i}\rightarrow X)(s)=s_{i}$ for all $i\in I$ if and only if for all $(i,j)\in I\times I$ we have $\mathcal{F}(U_{i}\times_{X}U_{j}\rightarrow U_{i})(s_{i})=\mathcal{F}(U_{i}\times_{X}U_{j}\rightarrow U_{j})(s_{j})$.
\end{defn}
The latter condition in this definition is a gluing condition familiar from the theory of sheaves on topological spaces.  Most often, we will take $\mathsf{C}=\mathsf{Sch}$, the category of schemes.
\begin{defn}[fppf/\'{e}tale Pretopology]\label{fppf}
The \emph{fppf} (resp. \emph{\'{e}tale}) pretopology on the category of schemes consists of coverings $\{U_i\to X\}_{i \in I}$, where $U_i \to X$ is a flat (resp. \'{e}tale) morphism locally of finite presentation in the category of schemes and $$\coprod_{i \in I} U_i \to X$$ is surjective.
\end{defn}

One can check that Definition \ref{fppf} satisfies Definition \ref{GP} of a Grothendieck pretopology.  Furthermore, we see once again that the topological aspects of super-geometry are solely encoded in the underlying bosonic topology.  We now obtain the following lemma, which has important consequences for our development of stacks.
\begin{lemma}\label{funcp}
The functor of points of a scheme is a sheaf with respect to the fppf topology.
\end{lemma}
\proof This result follows from definitions \ref{funcsheaf}, \ref{fppf}, and descent theory (c.f., \cite{Vistoli:2005}). $\Box$

\subsubsection*{\textit{Functors as Fibered Categories}}

One way of viewing stacks classically is as categories which are fibered in groupoids, and this is precisely the perspective we want to extend to the setting of super-geometry.  A groupoid is a group with several identities, or equivalently, as a set with automorphism groups attached to every element.  This will be the ideal way to describe orbifolds and stacks in algebraic super-geometry.
\begin{defn}[Groupoid]\label{groupoid}
A \emph{groupoid} is a small category, where all morphisms are isomorphisms.
\end{defn}

A groupoid-valued lax 2-functor from $\mathcal{F}:\mathsf{C}\rightarrow\mathsf{Grpd}$ is given by assigning to every object $X$ a groupoid $\mathcal{F}(X)$, and to every morphism $X\rightarrow Y$ a morphism of groupoids $\mathcal{F}(X)\rightarrow\mathcal{F}(Y)$.  This differs from an ordinary functor because the identity
\begin{equation*}
\mathcal{F}(g \circ f) =  \mathcal{F}(g) \circ \mathcal{F}(f)
\end{equation*}
is \emph{not} precisely satisfied; additionally, one requires a variety of other consistency conditions be satisfied.  Rather than following this track of 2-functors (which is somewhat tedious even in the classical sense), we pursue the slightly more conceptual language of fibered categories.  This allows us to avoid undue techinicalities in our discussion of super-geometry without sacrificing rigor.

\begin{defn}[Category Fibered in Groupoids]\label{fibered}
Let $\mathcal{F}:\mathsf{C}\rightarrow\mathsf{D}$ be a functor.  Let $y\in\ob(\mathsf{C})$ and $u\rightarrow v$ a morphism in $\mathsf{D}$.  If $v=\mathcal{F}(y)$, we say that a morphism $x\rightarrow y$ completes this data to a cartesian diagram if $\mathcal{F}(x)=u$ and for every $z\in\ob(\mathsf{C})$ with a morphism to $y$ and $\mathcal{F}(z)\rightarrow u$ making the diagram:
\[
\xymatrix{ \mathcal{F}(z) \ar[r] \ar[d] & \mathcal{F}(y) \ar@{=}[d] \\
u \ar[r] & v} 
\]
commute, there exists a unique morphism $z \to x$ such that
\begin{equation*}
\xymatrix{
\mathcal{F}(z) \ar@/_/[ddr] \ar@/^/[drr]
\ar@{-->}[dr] \\
& \mathcal{F}(x) \ar[d] \ar[r]
& \mathcal{F}(y) \ar@{=}[d] \\
& u \ar[r] & v }
\end{equation*}
commutes.  We say that $\mathcal{F}:\mathsf{C}\rightarrow\mathsf{D}$ is a \emph{category fibered in groupoids} (henceforth, a \emph{CFG}) if every diagram can be completed to a cartesian diagram.
\end{defn}

For every $u\in\ob(\mathsf{D})$, we consider the subcategory of $\mathsf{C}$ given by $x \in \ob(\mathsf{C})$ satisfying $\mathcal{F}(x) = u$ and morphisms $x \to y$ satisfying $\mathcal{F}(x \to y) = \mathrm{id}_u$.  This subcategory is called the fiber over $u$ of $\mathsf{C}$.  Definition \ref{fibered} states that a category fibered in groupoids is precisely a functor with all fibers being groupoids. In the language of $2$-functors (which we have neglected here), this gives rise to a lax 2-functor $(\mathsf{D})^{\mathrm{op}}\rightarrow\mathsf{Grpd}$ sending $u$ to the groupoid $\mathcal{F}^{-1}(u)$. The functoriality stems from the fact that we can pullback an object $y$ in $\mathsf{C}$ lying over $v$ along a morphism $\pi: u \to v$, to obtain $\pi^*y$. This $2$-functorial point of view suggests that there should exist an analogue of sheaves in line with Definition \ref{funcsheaf}.

The functor which we will need to satisfy a sheaf property in this context is the ``functor of isomorphisms,'' but we also demand a descent property for the objects in the CFG. Let $\mathcal{F}:\mathsf{C}\rightarrow\mathsf{D}$ be a CFG and $u \in \ob(\mathsf{D})$ an object. Given two objects $x$ and $y$ of $\mathsf{C}$ lying over $u$, we can define a set-valued functor $$\mathrm{Iso}(x,y): (\mathsf{D}/u)^{\mathrm{op}} \to \mathsf{Set}.$$ It sends an object of the comma category $\mathsf{C}/u$, which is simply a morphism $\phi: v \to u$, to the set of isomorphisms $$\phi^*x \to \phi^*y.$$ Note that this functor is well-defined up to equivalence, as the same is true for the pullbacks $\phi^*u$ and $\phi^*v$. We refer to subsection I.3.7 in \cite{Vistoli:2005} for a concise discussion of this concept.

\begin{defn}[Stack]\label{stacks}
Let $\mathcal{F}:\mathsf{C}\rightarrow\mathsf{D}$ be a CFG; and assume that $\mathsf{D}$ is endowed with a Grothendieck pretopology. We say that $\mathcal{F}$ is a \emph{stack}, if the following two conditions are satisfied:
\begin{itemize}
\item Given two objects $x$ and $y$ of $\mathsf{C}$, lying over the same object $u$ of $\mathsf{D}$; the functor of isomorphisms between $x$ and $y$ $$\mathrm{Iso}(x,y): (\mathsf{C}/u)^{\mathrm{op}} \to \mathsf{Set}$$ is a sheaf.

\item Let $\{\iota_i: u_{i}\rightarrow u\}_{i\in I}$ be a covering in $\mathsf{D}$, and $p_{1}:u_{i}\times_{i}u_{j}\rightarrow u_i$, $p_{12}:u_i \times_u u_j \times_u u_k\rightarrow u_{i}\times_{u}u_{j}$, etc. be the natural projections.  Then given objects $x_i$ of $\mathsf{C}$ lying over $u_i$, together with isomorphisms $$\phi_{ij}: p_{1}^*x_i \to p_2^*x_j$$ on $u_i \times_u u_j$ satisfying the cocycle condition $$p_{23}^*\phi_{jk} \circ p_{12}^*\phi_{ij} = p_{13}^*\phi_{ik}$$ on $u_i \times_u u_j \times_j u_k$, there exists an object $x$ over $u$ together with isomorphisms $$\phi_i: \iota_i^*x \to x_i$$ on $u_i$, satisfying $$\phi_{ij} \circ \phi_i = \phi_{j}$$ on $u_i \times_u u_j$.
\end{itemize}
\end{defn}

\subsubsection{\textit{Super-stacks}}

We are now ready to set out our superized notion of stacks.  To do this, we choose a Grothendieck pretopology on the category of schemes, $\mathsf{Sch}$, resulting in the following definition:
\begin{defn}[Super-stack]\label{superstack}
A \emph{super-stack} (henceforth, a \emph{stack}) is a CFG over the category of schemes which is a stack with respect to the \emph{fppf} pretopology.
\end{defn}

Just as $\tau_b$ was extended to a functor for schemes, we can extend it to a functor of stacks as well.  The motivation for such an extension is derived from the universal property of $\tau_b$ given by Lemma \ref{univ}, and how it acts on the space of maps between schemes.  We have also seen that the functor of points of the bosonic scheme $\tau_b(X)$ can be described as the functor of points of the scheme $X$ restricted to bosonic schemes.  Such a definition is sensible because any Grothendieck pretopology on $\mathsf{Sch}$ restricts (by Definition \ref{fppf}) to the usual pretopology on $\mathrm{B}\mathsf{Sch}$.

\begin{defn}\label{bred}
For a stack $\mathcal{X}$, we define $\tau_b(\mathcal{X})$ to be the bosonic stack given by restriction to the full subcategory of bosonic schemes.
\end{defn}
As before, $\tau_b$ is the right adjoint to the fully faithful inclusion functor between 2-categories $\iota: \mathrm{B}\mathsf{Stack}\rightarrow\mathsf{Stack}$.  Hence, to view a bosonic stack $\mathcal{X}_{b}$ as a stack over $\mathsf{Sch}$, we take
\begin{equation*}
\iota(\mathcal{X}_{b})(U)\equiv\mathcal{X}_{b}(U_{b}).
\end{equation*}

\subsubsection{\textit{Morphisms of Stacks and Algebraicity}}

Surjectivity is a property which will be important in the study of morphisms between stacks.  At the level of schemes, the condition that $Y\rightarrow X$ is surjective is just the condition that the morphism is surjective upon restriction to the underlying bosonic topological spaces (i.e., $\tau_b(Y)\rightarrow\tau_b(X)$ is surjective).  To define a notion of surjectivity for morphisms of stacks, we use the following procedure, well-known in the bosonic case \cite{Laumon:2000}, using the stack's set of points.

Let $\mathcal{X}$ be some stack over affine schemes and $F,L,K$ some fields, and consider the set of isomorphism classes $\mathcal{X}(\mathrm{Spec}(F,L,K))$.  We then identify a point $x\in\ob(\mathcal{X}(\mathrm{Spec}(F)))$ with $y\in\ob(\mathcal{X}(\mathrm{Spec}(L)))$ if there exists a common sub-field $K\subset F,L$ and a point $z\in\ob(\mathcal{X}(\mathrm{Spec}(K)))$ such that $x$ and $y$ both map to $z$.  The set of all such points in $\mathcal{X}\rightarrow\mathsf{Sch}$ up to this identification is called the set of points of the stack $\mathcal{X}$. This is a well-known concept from the theory of bosonic stacks, and only depends on $\tau_b(\mathcal{X})$ \cite{Laumon:2000}. Every morphism of stacks induces a map between their sets of points, and consequently we can call a morphism of stacks surjective if and only if the corresponding map between sets of points is surjective.  In particular, this means that the morphism $\mathcal{Y}\rightarrow\mathcal{X}$ is surjective if and only if the morphism of bosonic stacks $\tau_b(\mathcal{Y})\rightarrow\tau_b(\mathcal{X})$ is surjective. 

\begin{defn}[Schematic morphism]\label{schematic}
A morphism of algebraic stacks $f:\mathcal{Y}\rightarrow\mathcal{X}$ is said to be \emph{schematic} if, for all $X\rightarrow\mathcal{X}$ (where $X$ is a scheme), the diagram
\begin{equation*}
\xymatrix{
Y \ar[d] \ar[r] & X \ar[d] \\
\mathcal{Y} \ar[r]^{f} & \mathcal{X}
}
\end{equation*}
is cartesian for some scheme $Y$. We say that a schematic morphism is smooth, \'etale, flat or bosonic, if every base change $Y \to X$ has the respective property.
\end{defn}

The notion of schematic morphisms allows us to define algebraic super-spaces. This concept has been introduced by \cite{DominguezPerez:1997rh}.

\begin{defn}[Algebraic Space]\label{algebra}
A set-valued sheaf $\mathcal{X}$ on the category of schemes is called an \emph{algebraic super-space} (henceforth, an \emph{algebraic space}), if there exists a scheme $S$ together with a morphism $S \to \mathcal{X}$, which is a surjective morphism of stacks, which is schematic and \'etale.
\end{defn}

In the same way that an ordinary functor can be represented by a scheme, we can define what it means for a morphism of algebraic stacks to be representable:
\begin{defn}[Representable morphism]\label{rep}
A morphism of stacks $f:\mathcal{Y}\rightarrow\mathcal{X}$ is said to be \emph{representable} if, for all $X\rightarrow\mathcal{X}$ (where $X$ is a scheme), the diagram
\begin{equation*}
\xymatrix{
Y \ar[d] \ar[r] & X \ar[d] \\
\mathcal{Y} \ar[r]^{f} & \mathcal{X}
}
\end{equation*}
is cartesian for some algebraic space $Y$.
\end{defn}
If a morphism of stacks is representable, then we are free to define what is meant by other local morphism properties for it.  In this manner, any property of schemes which is invariant under base change can be defined as a property of morphisms of stacks, provided that morphism is representable; examples are \emph{\'{e}tale, smooth, un-ramified}, and so on.  The following definition sets out formal terminology which we use to deal with morphisms of stacks from now on:

\begin{defn}\label{propmorphs}
A morphism of stacks $\mathcal{Y}\rightarrow\mathcal{X}$ is called:
\begin{enumerate}

\item \emph{bosonic} if every base change of  $\mathcal{Y}\rightarrow\mathcal{X}$ in definition \ref{rep} is bosonic,

\item \emph{smooth} if every base change of  $\mathcal{Y}\rightarrow\mathcal{X}$ in definition \ref{rep} is smooth,

\item \emph{\'{e}tale} if every base change of $\mathcal{Y}\rightarrow\mathcal{X}$ in definition \ref{rep} is \'{e}tale.
\end{enumerate}
\end{defn}

With this definition at hand we can finally introduce algebraic stacks:

\begin{defn}[Algebraic Stack]\label{algss}
A stack $\mathcal{X}$ is called \emph{algebraic} if there exists an algebraic space $S$ together with a surjective, representable and smooth morphism $S \to \mathcal{X}$.
\end{defn}

We will say that the morphism $S \to \mathcal{X}$ is a \emph{complete versal family} for the algebraic stack $\mathcal{X}$\footnote{Often this morphism is referred to as an \emph{atlas} for the stack $\mathcal{X}$}. This terminology is taken from \cite{Behrend:2011} and emphasizes that in general stacks should be viewed as abstract moduli problems. Morally speaking, moduli problems are about classifying certain objects up to isomorphism. A complete versal family is a continuously (or regularly) varying family of such objects, such that every isomorphism type is covered (this is the surjectivity or completeness condition); versality (i.e., smoothness) is a weakened universality condition, which ensures that every other family can be obtained \'etale locally from our chosen one.

Now that we have a fully-developed vocabulary for describing morphisms between stacks in super-geometry, we are also ready to define a Deligne-Mumford stack:

\begin{defn}[Deligne-Mumford stack]\label{DM}
A stack $\mathcal{X}\rightarrow\mathsf{Sch}$ is \emph{Deligne-Mumford} if there exists a complete versal family $S$ with the corresponding morphism $S \to \mathcal{X}$ being \'{e}tale.
\end{defn}

We obtain the following fact directly from our definitions:
\begin{lemma}\label{ADM}
Let $\mathcal{Y}\rightarrow\mathcal{X}$ be a representable morphism of stacks.  If $\mathcal{X}$ is an algebraic stack (resp. Deligne-Mumford stack) then so is $\mathcal{Y}$.
\end{lemma}
\proof  If $\mathcal{X}$ is algebraic, then it contains a complete versal family $S \to \mathcal{X}$ for some scheme $S$.  By the definition of a representable morphism of stacks, we know that the base change $\mathcal{Y} \times_{\mathcal{X}} S \to \mathcal{Y}$ is a complete versal family of $\mathcal{Y}$; further, this will be \'{e}tale if the complete versal family of $\mathcal{X}$ was \'{e}tale.     $\Box$

%%%%%%%%%%%%%%%%%%%%%%%%%%%%%%%%%%%%%%%%%%%%%%%%%%%%%%%%%%%%%%%%%%%%
%%%%%%%%%%%%%%%%%%%%%%%%%%%%%%%%%%%%%%%%%%%%%%%%%%%%%%%%%%%%%%%%%%%%

\section{The Moduli Stack of Stable Maps}
\label{Moduli}

Having developed our understanding of algebraic super-geometry to the extent that we have a suitable notion of algebraic and Deligne-Mumford stacks, we now proceed to construct the object which will be of interest for twistor-string theory.  In what follows, we will consider $X$ to be a (complex) smooth projective scheme; this will be the target space of our string theory.  After providing the definitions necessary to construct the moduli stack of interest, we prove that it is a Deligne-Mumford stack.

\subsection{Moduli of Stable Maps}

Our first step must be to find a supersymmetric generalization of Kontsevich's concept of a stable map from a Riemann surface into a variety or bosonic scheme \cite{Kontsevich:1995}.  This proves to be a rather trivial generalization of the usual definition:

\begin{defn}[Stable Map]\label{Stablemap}
A \emph{stable map} over $T\in\ob(\mathsf{Sch})$ into $X$ is given by the following set of data:
\begin{equation*}
\left\{ \pi: \mathcal{C}\rightarrow T, \;  g, \; n,\; \beta\in H_{2}(X,\Z),\; \phi:\mathcal{C}\rightarrow X\right\},
\end{equation*}
where $\mathcal{C}$ is an algebraic space, $\pi$ is a proper, flat, bosonic morphism whose geometric fibers $\mathcal{C}_{t}$ are reduced, connected, and one-dimensional bosonic schemes (i.e. possibly singular Riemann surfaces), and $n=\{x_{i}: T\rightarrow\mathcal{C}\}_{i=1,\ldots,n}$ are marked points which vary smoothly between the fibers of $\mathcal{C}$.  Furthermore, these fibers obey the following conditions:
\begin{enumerate}
\item $\dim H^{1}(\cO_{\mathcal{C}_{t}})=g$;

\item the only singularities of $\mathcal{C}_{t}$ are ordinary double points;

\item every contracted irreducible component of arithmetic genus $h$ of $\mathcal{C}_t$ contains at least $3 - 2h$ special (i.e. marked or singular) points on its normalization;

\item $\phi_*[C_t] = \beta$.
\end{enumerate}
\end{defn}
In this definition, the notion of homology on a projective scheme $X$ has to be understood as the homology of the analytic space associated to $\tau_b(X)$. Criterion (\emph{3}.) of this definition is just the requirement that the automorphism group of the map be finite; this is the well-known hallmark of stability for curves and maps.

For a morphism of schemes $u:S\rightarrow T$, it is possible to pullback a $T$-family of stable maps to an $S$-family of stable maps. This allows us to organize families of stable maps into a CFG over $\mathsf{Sch}$.

\begin{defn}\label{Modspace}
Let $\bar{\cM}(X,g,n,\beta)(T)$ be moduli stack of maps into $X$ over $T$, as specified in Definition \ref{Stablemap}. Letting the base scheme $T$ vary, we obtain a CFG over $\mathsf{Sch}$ which is the moduli stack of stable maps to $X$, denoted by $\bar{\cM}_{g,n}(X,\beta)$.
\end{defn}

The following is a tautology and is proved by applying Lemma \ref{univ} at the level of each family in the moduli stack.

\begin{lemma}\label{taubM}
The bosonic truncation of $\bar{\cM}_{g,n}(X,\beta)$ is canonically equivalent to the moduli stack $\bar{\cM}_{g,n}(\tau_{b}(X),\beta)$; i.e., $$\tau_b(\bar{\cM}_{g,n}(X,\beta)) \cong \bar{\cM}_{g,n}(\tau_b(X),\beta). $$
\end{lemma}

\subsection{Stack Properties}

In Definition \ref{Modspace} we have introduced a CFG over the category of schemes which describes the moduli problem of stable maps; our next goal is to prove a result analogous to the classic theorem of Behrend and Manin in the bosonic setting \cite{Behrend:1995}, namely that this CFG is in fact an algebraic stack.  For this we require that the target scheme $X$ be globally split; that is, there exists a locally free sheaf $\mathcal{V}$ on $\tau_b(X)$, such that $X$ is equivalent to the relative spectrum of the sheaf of algebras $\bigwedge^{\bullet}\mathcal{V}$. 
\begin{thm}\label{Main}
Let $X$ be a smooth, projective, and split scheme. Then $\bar{\cM}_{g,n}(X,\beta)\rightarrow\mathsf{Sch}$ is a Deligne-Mumford stack.
\end{thm}

\proof  We divide the proof of this theorem into several smaller steps.  To avoid a proliferation of notation where it is not needed, we denote $\bar{\cM}_{g,n}(X,\beta)$ by $\bar{\cM}(X)$ throughout this proof, and often suppress the notation for the image class, genus, and marked points when discussing stable maps.  The assumption that $X=(X_{b},\mathfrak{A})$ is a globally split scheme means that there exists a locally free sheaf $\mathcal{V}$ on $\tau_b(X)$ such that $\mathfrak{A}=\bigwedge^{\bullet}\mathcal{V}$ (i.e., the super-structure on $X$ is determined by an an exterior algebra).  We begin with the following observation about families of stable maps:
\begin{lemma}\label{Main1}
Let $\{\pi: \mathcal{C}\rightarrow T, x_i: T \to \mathcal{C} ,\phi: \mathcal{C}\rightarrow X\}$ be a family of stable maps parameterized by $T\in\ob(\mathsf{Sch})$.  Then $\{\mathcal{C}_{b}\rightarrow T_{b}, x_i: T_b \to \mathcal{C}_b,\mathcal{C}_{b}\rightarrow X_{b}\rightarrow \tau_b(X)\}$ is also a family of stable maps.
\end{lemma}
\proof  Recall from Definition \ref{Stablemap} that stability is determined by the fibers of the family having finite automorphism group; this is a purely bosonic property, so it follows that $\{\mathcal{C}_{b}\rightarrow T_{b}, \mathcal{C}_{b}\rightarrow X_{b}\rightarrow \tau_b(X)\}$ is a family of stable maps in the bosonic sense of Kontsevich \cite{Kontsevich:1995} provided the underlying morphism $\mathcal{C}_{b}\rightarrow T_{b}$ is flat.  But by definition $\pi: \mathcal{C}\rightarrow T$ is a flat, bosonic morphism, so by Proposition \ref{flat3}, the result follows.     $\blacksquare$

The bosonic reduction sending any family of stable maps to the underlying bosonic family of stable maps induces a functor between categories which we will interpret as a morphism of CFGs: $\Xi: \bar{\cM}(X)\rightarrow\tau_b(\bar{\cM}(X))$.  Furthermore, $\tau_b(\bar{\cM}(X))=\bar{\cM}(\tau_b(X))$ as was stated in Lemma \ref{taubM}. 

\begin{lemma}\label{Main4}
The morphism $\Xi:\bar{\cM}(X)\rightarrow\bar{\cM}(\tau_b(X))$ is representable.
\end{lemma}
\proof  Let $T$ and $W$ be affine schemes and consider the following set-up:
\begin{equation*}
\xymatrix{
W \ar@/_/[ddr] \ar@/^/[drr] \ar@{-->}[dr] \\
& P \ar[d] \ar[r] & T \ar[d] \\
& \bar{\cM} \ar[r]^{\Xi} & \tau_b(\bar{\cM}) }
\end{equation*}
where $W$ is the test scheme for the pullback and the square is cartesian.  As $\tau_b(\bar{\cM})$ is bosonic, it follows that the morphism $T\rightarrow\tau_b(\bar{\cM})$ is fully captured by $T_{b}\rightarrow\tau_b(\bar{\cM})$, and identical statements can be made for every other morphism or composition of morphisms to $\tau_b(\bar{\cM})$ appearing in this diagram.  To see this, use Lemma \ref{flat2}, which implies that $$\mathcal{C} \cong \mathcal{C}_b \times_{T_b} T.$$ Therefore the data of a family of curves parametrized by $T$ is equivalent to a family of curves parametrized by $T_b$.  Similarly the sections $$x_i: T \to \mathcal{C}$$ can be recovered from their bosonic part. As we have seen the only part of a family of stable maps amenable to the super-structure on $X$ is the map $\phi: \cC \to X$. Henceforth we may assume that $T$ is bosonic.

We thus have a bosonic scheme $T$ and a family of stable maps $\{\pi: \cC \to T, x_i: T \to \cC ,\phi: \cC \to \tau_b(X)\}$, which we pullback to a family of stable maps into $\tau_b(X)$ parametrized by $W_b$. In order to compute the fibre product in question we need to describe in which ways this family of stable maps can be extended to a stable map into $X$. If $\mathcal{R}$ denotes the fermionic component of the sheaf of super-rings on $W$, and $X = \mathrm{Spec}_{\tau_{b}(X)}(\bigwedge^{\bullet} \mathcal{V})$ then those extensions are given by
$$ \pi_*(\phi^*\mathcal{V}^{\vee} \otimes \pi^*\mathcal{R}). $$
Here we have implicitly used Lemma \ref{flat2}, since $\pi^*\mathcal{R}$ is the fermionic component of the family of curves parametrized by $W$. We then apply 7.6 of \cite{EGA:3} to see that there exists a coherent sheaf $\mathcal{Q}$ on $T$ such that for every quasi-coherent sheaf $\mathcal{R}$ on $T$,
\begin{equation}\label{eqn:Q}
\pi_{*}\left((\phi^{*}\mathcal{V})^{\vee}\otimes\pi^{*}\mathcal{R}\right)=\mathcal{H}\mathrm{om}_{T}(\mathcal{Q},\mathcal{R}).
\end{equation}
One can show that this construction is compatible with base change (see remark 7.9 of \cite{EGA:3}).  Then using $\mathcal{Q}$ as the generator for an exterior algebra, we see that morphisms $\phi^*\mathcal{V} \to \pi^*\mathcal{R}$ correspond to $\wedge^{\bullet}\mathcal{Q} \to \mathcal{R}$, giving the required universal property and completing the proof. $\blacksquare$

(\emph{End of proof of Theorem \ref{Main}}).  By assumption, $\tau_b(\bar{\cM}(X))=\bar{\cM}(\tau_b(X))$ is a Deligne-Mumford stack, and by construction it is easy to see that $\bar{\cM}(X)$ is a CFG over the category of schemes.  Then using the fact that $\Xi:\bar{\cM}(X)\rightarrow\bar{\cM}(\tau_b(X))$ is representable from Lemma \ref{Main4}, it follows by Lemma \ref{ADM} that $\bar{\cM}(X)$ is also a Deligne-Mumford stack, as required.     $\Box$

\bigskip

The defining equation \ref{eqn:Q} of the sheaf $\mathcal{Q}$ on $\bar{\cM}_{g,n}(X,\beta)$ allows us to make the following observation:

\begin{lemma}\label{Main5}
$\mathcal{Q}^{\vee} = \pi_*(\phi^*\mathcal{V}^{\vee})$
\end{lemma}

\begin{proof}
This follows from setting $\mathcal{R} = \mathcal{O}$ in equation (\ref{eqn:Q}). $\Box$
\end{proof}

Since $\bar{\cM}_{g,n}(\tau_b(X),\beta)$ is a bosonic moduli stack of stable maps, there exists a universal curve $(\widetilde{\Sigma},\widetilde{\underline{n}})$ and a \emph{universal instanton}:
\be{eqn: bunivinst}
\xymatrix{
 (\widetilde{\Sigma},\widetilde{\underline{n}}) \ar[d]^{\rho} \ar[r]^{\Phi} & \tau_b(X) \\
 \bar{\cM}_{g,n}(\tau_b(X),\beta) & }
\ee 
The universal instanton construction allows us to pull back geometric structures from the target $\tau_b(X)$ and then push them down onto the moduli stack. A similar construction was unravelled in Lemma \ref{Main4}. The maps $\pi$ and $\phi$ in this Lemma are obtained from the universal instanton maps $\rho$ and $\Phi$ by base change. In particular we see that the stack $\bar{\cM}_{g,n}(\tau_b(X),\beta)$ is endowed with a sheaf $\mathcal{Q} = (\rho_*(\Phi^*\mathcal{V}^{\vee}))^{\vee}$, which generates an exterior algebra $\wedge^{\bullet} \mathcal{Q}$, giving rise to the moduli stack $\bar{\cM}_{g,n}(X,\beta)$ by the relative spectrum construction. We may therefore conclude that $\bar{\cM}_{g,n}(X,\beta)$ is a split Deligne-Mumford super-stack.

From this result and the known theorems for the bosonic stack (Theorem 3.14 of \cite{Behrend:1995}), we get an easy corollary which addresses the cases of interest in twistor theory:
\begin{corol}\label{BMc}
Let $X$ be any split smooth projective $\C$-scheme, then $\bar{\cM}_{g,n}(X,\beta)$ is a split Deligne-Mumford stack. In particular, $$\bar{\cM}_{g,n}(\P^{p|q},\beta)$$ is a split Deligne-Mumford stack.
\end{corol}

As before we know that $\bar{\cM}_{g,n}(X,\beta)$ is a moduli stack of stable maps, and thus there exists a universal curve $(\widetilde{\Sigma},\widetilde{\underline{n}})$ and a \emph{universal instanton} in the super-context as well:
\be{eqn: univinst}
\xymatrix{
 (\widetilde{\Sigma},\widetilde{\underline{n}}) \ar[d]^{\rho} \ar[r]^{\Phi} & X \\
 \bar{\cM}_{g,n}(X,\beta) & }
\ee 

%%%%%%%%%%%%%%%%%%%%%%%%%%%%%%%%%%%%%%%%%%%%%%%%%%%%%%%%%%%%%%%%%%

\subsection{Other properties of the moduli stack}

In this section, we assume that Theorem \ref{Main} holds, so $\bar{\cM}_{g,n}(X,\beta)$ is a Deligne-Mumford stack; as noted by Corollary \ref{BMc}, this will be true for most schemes $X$ which arise in physical applications (i.e., split projective smooth $\C$-schemes).  The properties of the bosonic stack of stable maps to a bosonic scheme are well-studied, and the various properties and underlying structures of this space are by now well-known (c.f., \cite{Fulton:1997,Cox:1999}).  Most of these properties carry over without change to the super-geometric setting, and we review them here briefly.

Let $\Sigma$ be some Riemann surface of genus $g$; recall from definition \ref{Stablemap} the notion of a stable map $\phi:\Sigma\rightarrow X$ (for some scheme $X$). More formally, we can think of this as restricting our attention to a $\C$-family of stable maps in $\bar{\cM}_{g,n}(X,\beta)$. We represent this single object in $\bar{\cM}_{g,n}(X,\beta)$ by $(\Sigma, \underline{n}, \phi)$, where $\underline{n}$ is shorthand for the set of special points $\{x_{1},\ldots,x_{n}\}$ on $\Sigma$.

The moduli stack $\bar{\cM}_{g,n}(X,\beta)$ comes equipped with several natural maps.  These include the ``evaluation maps''
\be{eqn: eval}
\mathrm{ev}_{i}: \bar{\cM}_{g,n}(X,\beta)\rightarrow X, \qquad (\Sigma, \underline{n},\phi)\mapsto \phi(x_{i}),
\ee
which can be tensored together in the obvious fashion to give
\be{eqn: evaln}
\mathrm{Ev}:\bar{\cM}_{g,n}(X,\beta)\rightarrow X^{n}.
\ee
Since $\Sigma$ is a bosonic Riemann surface of genus $g$, the underlyings stacks $\bar{\cM}_{g}$ and $\bar{\cM}_{g,n}$ (stable curves of genus $g$ and stable curves of genus $g$ with $n$ marked points, respectively) are the ordinary bosonic stacks of Deligne and Mumford \cite{Deligne:1969}, with dimensions $3g-3$ and $3g-3+n$ respectively.  Since $(\Sigma,\underline{n},\phi)$ need not be a stable curve on its own, we can define a projection to $\bar{\cM}_{g,n}$ in the usual way: provided $n+2g\geq 3$, simply contract the destabilizing components of $\Sigma$ to obtain a stable curve $\widehat{\Sigma}$.  This defines a functor
\be{eqn: curvfunc}
\kappa: \bar{\cM}_{g,n}(X,\beta)\rightarrow \bar{\cM}_{g,n}, \qquad (\Sigma,\underline{n},\phi)\mapsto (\widehat{\Sigma},\underline{n}).
\ee
When $\bar{\cM}_{g,n}(X,\beta)$ posesses a coarse moduli space, this should descend to a morphism following the techniques used in the bosonic setting \cite{Knudsen:1983}.

We also have the forgetful functor
\be{eqn: forget}
\rho_{n}: \bar{\cM}_{g,n}(X,\beta)\rightarrow\bar{\cM}_{g,n-1}(X,\beta).
\ee
This functor is defined as in the bosonic setting, where it inherits its structure from the underlying functor between bosonic stacks $\bar{\cM}_{g,n}\rightarrow\bar{\cM}_{g,n-1}$, which forgets the marked point $x_{n}$ and contracts any resulting destabilizing components of $\Sigma$.  The forgetful functors are, of course, well-defined only when both the source and target in \eqref{eqn: forget} exist.  

An important property of bosonic stacks is the fact that they can have impure dimension: deformations of the moduli stack can be obstructed and the dimension of the space can change when points in a family are obstructed.  Nevertheless, one can still compute the expected or \emph{virtual} dimension of the moduli stack, which corresponds to the dimension of the space when the deformation theory is unobstructed.  We can apply the tangent-obstruction complex techniques of \cite{Li:1998}, along with a super-geometric generalization of the Hirzebruch-Riemann-Roch theorem to find for the virtual super-dimension
\be{eqn: vdim1}
\mathrm{vsdim} \; \bar{\cM}_{g,n}(X,\beta)=(1-g)(\sdim X-3)-\int_{\beta}\omega_{X}+n,
\ee
where $\sdim X$ is the super-dimension of $X$ and $\omega_{X}$ is the canonical class of $X$.  Recall that when $X$ is a split scheme, $X=(X_{b},\wedge^{\bullet}\mathcal{V})$ for a locally free sheaf $\mathcal{V}$ on $\tau_b(X)$.  In this case, $\sdim X=\dim \tau_b(X)-\mathrm{rank}\; \mathcal{V}$.

As in the bosonic case, the dimension of $\bar{\cM}_{0,n}(X,\beta)$ is equal to the expected (virtual) dimension when $H^{1}(C,\phi^{*}T_{X})=0$ for all genus zero stable maps $\phi$; this means that the deformation theory of the moduli stack is unobstructed \cite{Kontsevich:1994qz, Fulton:1997}.  We will use the same terminology as in the bosonic category, and refer to such target schemes $X$ as \emph{convex}; an easy example of a convex scheme is the projective space $\P^{p|q}$.

In the bosonic category, the moduli stack will be smooth when we restrict to genus zero and the the target scheme is convex.  The machinery from the proof of Theorem \ref{Main} allows us to make an analogous statement in our super-geometric setting:

\begin{propn}\label{Main2}
Let $X$ be a split smooth projective scheme which is convex. Then $\bar{\cM}_{0,n}(X,\beta)$ is a split, smooth Deligne-Mumford stack; in particular, $$\bar{\cM}_{0,n}(\P^{p|q},\beta)$$ is smooth.
\end{propn}

\begin{proof}
Let $X$ be equivalent to the relative spectrum of $\wedge^{\bullet} \mathcal{V}$ on $\tau_b(X)$, where $\mathcal{V}$ is a locally free sheaf on $\tau_b(X)$. The convexity condition is now equivalent to $\tau_b(X)$ being convex and for every map $\phi: \P^1 \to \tau_b(X)$ we have $H^1(\P^1,\phi^*\mathcal{V}^{\vee}) = 0$. To see this one observes that $\mathcal{V}^{\vee}$ is the fermionic part of the tangent sheaf $T_X$. 

Let $S$ be a versal family for the stack $\bar{\cM}_{0,n}(\tau_b(X),\beta)$. In the proof of  Lemma \ref{Main4} we are able to set $\mathcal{R} = \mathcal{O}_{s}$ for every point $s \in S$, as in Lemma \ref{Main5}; note that $\mathcal{O}_s$ is meant to be the structure sheaf of the point $s$. We see that $(\mathcal{Q} \otimes\mathcal{O}_s)^{\vee}$ is the same as $H^0(\pi^{-1}(s),\phi^*\mathcal{V}^{\vee})$ from equation \ref{eqn:Q}. The Riemann-Roch formula and the convexity condition for the fermionic part implies now that the rank of this vector space is constant. We conclude that $\mathcal{Q}$ is a coherent sheaf on $S$ of constant rank. The convexity condition for the bosonic part implies that $\bar{\cM}_{0,n}(\tau_b(X),\beta)$ is smooth, and a coherent sheaf of constant rank on a reduced algebraic space is locally free. This implies that $\bar{\cM}_{0,n}(X,\beta)$ is a smooth Deligne-Mumford stack.   $\Box$
\end{proof}

\subsubsection*{\textit{Example: Witten's construction for} $\bar{\cM}_{0,0}(\P^{p|q},d)$}

For a simple reality check on our formula for the virtural super-dimension of the moduli stack, we consider $\bar{\cM}_{0,0}(\P^{p|q},\beta)$.  As in the bosonic case, $H_{2}(\P^{p|q},\Z)\cong\Z$, so we can write $\beta=d[\ell]$, where $[\ell]$ is the class of a line and $d$ is the degree of the stable map.  This allows us to abbreviate $\bar{\cM}_{g,n}(\P^{p|q},\beta)$ by $\bar{\cM}_{g,n}(\P^{p|q},d)$.  We now review Witten's \cite{Witten:2003nn} construction of a versal family for $\bar{\cM}_{0,0}(\P^{p|q},d)$ on a dense open subset.  We adopt a very heuristic view, treating a stable map to $\P^{p|q}$ as a map from $\P^1$; in reality, this should be tensored with some super-ring so as not to violate Lemma \ref{univ}, but we ignore these subtleties here.

The basic idea is to construct a simple versal family of stable maps for a dense open subset of the moduli space.  On $\P^{p|q}$ choose homogeneous coordinates $Z^{I}=(Z^{1},\ldots,Z^{p+1},\psi^{1},\ldots,\psi^{q})$, and let $\sigma=(\sigma_{1},\sigma_{2})\in\P^1$ be homogeneous coordinates on our genus zero Riemann surface.  Away from the boundary divisor in $\bar{\cM}_{0,0}(\P^{p|q},d)$ (i.e., for irreducible curves only) a degree $d$ map $Z^{I}:\P^{1}\rightarrow\P^{p|q}$ can be written as:
\begin{equation*}
Z^{I}(\sigma)=\sum_{r=0}^{d}U^{I}_{r}\sigma^{r}.
\end{equation*}
\textit{A priori}, the moduli of such a map are the coefficients $\{U_{r}^{I}\}$, which span a linear superspace:
\begin{equation*}
\mathrm{span} \{U_{r}^{I}\}=\mathbb{L}\cong\C^{(p+1)d+p+1|qd+q}.
\end{equation*}
Since the $Z^{I}$ are homogeneous coordinates, we must account for the re-scalings $Z^{I}\rightarrow tZ^{I}$ for $t\in\C^{*}$.  This reduces $\mathbb{L}$ to a projective linear space $\P\mathbb{L}\cong\P^{(p+1)d+p|qd+q}$.  Additionally, the map cannot vanish since the $Z^{I}$ are homogeneous, so we must cut out those $\{U_{r}^{I}\}$ which correspond to the zero locus.  Finally, we must account for the automorphism group of the Riemann surface, $\mathrm{PGL}(2,\C)$.  

Hence, we are left with an open subset of a projective super-space which can be identified with a versal family for the stack away from the boundary divisor and zero-locus (i.e., on a dense open subset):
\be{eqn: wms}
\bar{\cM}_{0,0}(\P^{p|q},d)\supset\cM_{0,0}(\P^{p|q},d) \cong (\P\mathbb{L} - \mbox{Zero Locus})/\mathrm{PGL}(2,\C).
\ee
From this we can immediately read off the expected super-dimension: $p-q-3+d(p-q+1)$.  Note that the stack properties are immediately obvious from the explicit presence of the $\mathrm{PGL}(2,\C)$-quotient.

Now, consider our formula for the virtual super-dimension \eqref{eqn: vdim1}:
\begin{equation*}
\mathrm{vsdim} \bar{\cM}_{0,0}(\P^{p|q},d)=\sdim\P^{p|q}-3-\int_{d[\ell]}\omega_{\P^{p|q}}.
\end{equation*}
We know that the canonical sheaf of $\P^{p}$ is $\cO(-p-1)$, and from \eqref{eqn: ss} $\mathfrak{A}_{\P^{p|q}}=\wedge^{\bullet}\cO(1)^{\oplus q}$, so
\begin{equation*}
-\int_{d[\ell]}\omega_{\P^{p|q}}=d(p+1)-dq,
\end{equation*}
and our formula for $\mathrm{vsdim} \bar{\cM}_{0,0}(\P^{p|q},d)$ agrees precisely with what is predicted by \eqref{eqn: wms}.

%%%%%%%%%%%%%%%%%%%%%%%%%%%%%%%%%%%%%%%%%%%%%%%%%%%%%%%%%%%%%%%%%%%%%%%
%%%%%%%%%%%%%%%%%%%%%%%%%%%%%%%%%%%%%%%%%%%%%%%%%%%%%%%%%%%%%%%%%%%%%%%

\section{Discussion \& Conclusion}
\label{Conclusion}

In this paper, we have constructed a Deligne-Mumford moduli stack of stable maps from a Riemann surface to a scheme in the context of super-geometry, and studied its properties.  On a purely mathematical level, this investigation re-emphasizes the utility of stacks for representing moduli problems and also demonstrates the extent to which super-geometric objects inherit many of their characteristics from the underlying bosonic geometry.  Indeed, Theorem \ref{Main} demonstrates that $\bar{\cM}_{g,n}(X,\beta)$ is Deligne-Mumford whenever $X$ is split projective and smooth, and Proposition \ref{Main2} shows that $\bar{\cM}_{0,n}(X,\beta)$ is smooth under conditions inherited from the bosonic case.  However, the machinery developed in Section \ref{SEGA} illustrates that the inclusion of super-geometric objects in algebraic geometry is by no means trivial.

The main assumption we have made in proving these results is that the target scheme $X$ is split.  While this covers a large number of physically interesting cases (including the maximally supersymmetric twistor space for four dimensional space-time, $\P^{3|4}$), there are some important examples which are left out.  In particular, general complex flag spaces are not generically split, and these are important in defining twistor geometry in higher dimensions (c.f., \cite{Baston:1989, Manin:1997}): recent investigations of six-dimensional gauge theories via twistor methods could lead to twistor-string-like developments in these more general settings \cite{Mason:2011nw, Saemann:2011nb, Saemann:2012rr}.  While we expect that our results should extend to the non-split category, it is clear that our strategy of proof (in particular, the methods used for Lemma \ref{Main4}) will not work.  It may be possible to cover such cases by working with a suitably abstract formalism: a generalization of Lurie's criteria for representability \cite{Lurie} to super-stacks could suffice, but we leave it to future research to investigate this issue in detail.  

In \cite{Movshev:2006py}, the problem of determining the Berezinian on the moduli space of smooth genus $g$ curves in $\P^{p|p+1}$ was studied from a more heuristic point of view whereby it was assumed that the moduli space was an orbifold.  The Berezinian sheaf could then be constructed via a pullback from $\bar{\cM}_{g,n}$ using the map $\kappa$ from \eqref{eqn: curvfunc}. It would be interesting to study to what extent these results generalize to the boundary of the moduli stack.

There are several other interesting questions which remain unanswered in this work: the status of Gromov-Witten theory for our moduli stacks; the existence of coase moduli schemes for these stacks; and the potential for applications in physics.  We will now say a few words about each of these issues, leaving it to future research to investigate them fully.

%%%%%%%%%%%%%%%%%%%%%%%%%%%%%%%%%%%%%%%%%%%%%%%%%%%%%%%%%%%%%%%%%%%%%%

\subsection{Gromov-Witten Theory}

In the bosonic category, one of the most important applications of Kontsevich's moduli space of stable maps is in the study of Gromov-Witten theory.  Here, Gromov-Witten invariants (rational numbers which can be interpreted as ``counting maps'' from a Riemann surface to the target scheme of a given genus and image class) are computed by integrating cohomology classes over the moduli stack.  In the context of super-geometry, the definition of Gromov-Witten invariants should be the same as in the bosonic category.  

Let $X$ be a smooth split projective scheme; we want Gromov-Witten invariants $\la I_{g,n,\beta}\ra$ to act as
\begin{equation*}
\la I_{g,n,\beta}\ra : H^{*}(X,\Q)^{\otimes n}\rightarrow \Q
\end{equation*}
via an integration over the moduli stack $\bar{\cM}_{g,n}(X,\beta)$.  This requires some homology cycle representing the moduli stack which we can integrate.  In the bosonic case, when $X$ is convex and $g = 0$, then there is a fundamental class $[\bar{\cM}_{0,n}(X,\beta)]$ which corresponds to $1\in H^{*}(\bar{\cM}_{0,n}(X,\beta), \Q)$ in the stack cohomology \cite{Behrend:2004a, Behrend:2004b} by Poincar\'e duality.  However, when $X$ is not convex or $g > 0$, one requires a virtual fundamental class $[\bar{\cM}_{g,n}(X,\beta)]^{\mathrm{virt}}$.  We assume that this object can also be defined using the machinery from the bosonic category: the perfect tangent-obstruction complex \cite{Li:1998}, or perfect obstruction theory \cite{Behrend:1997b,Behrend:1997a}.  Below, we will propose a definition using a formula from \cite{Kontsevich:1995}, which takes the fermionic part into account. 

Gromov-Witten invariants are then defined in the usual fashion (e.g., \cite{Cox:1999}):
\begin{defn}[Gromov-Witten invariants]\label{GWI}
Let $X$ be a smooth globally split projective scheme, $\beta\in H_{2}(X,\Z)$, and $\alpha_{1},\ldots,\alpha_{n}\in H^{*}(X,\Q)$.  For $g,n\geq 0$, the \emph{Gromov-Witten invariant} $\la I_{g,n,\beta}\ra(\alpha_{1}\ldots,\alpha_{n})$ is given by:
\be{eqn: GWId}
\la I_{g,n,\beta}\ra(\alpha_{1}\ldots,\alpha_{n})=\int\limits_{[\bar{\cM}_{g,n}(X,\beta)]^{\mathrm{virt}}} \mathrm{ev}_{1}^{*}(\alpha_{1})\cup \cdots \cup \mathrm{ev}_{n}^{*}(\alpha_{n}),
\ee
where $[\bar{\cM}_{g,n}(X,\beta)]^{\mathrm{virt}}$ is the virtual fundamental class and $\mathrm{ev}_{i}$ are the evaluations maps from \eqref{eqn: eval}.
\end{defn}

Since this definition is identical to the one from the bosonic category, and $\bar{\cM}_{g,n}(X,\beta)$ is just a stack over the category of schemes, it makes sense to assume that Gromov-Witten invariants in our super-geometric setting will have many of the same properties as their bosonic counterparts; in particular, they should still obey some super-geometric analogues of the axioms of Kontsevich and Manin \cite{Kontsevich:1994qz}.  However, one could worry that no new information is gained in the super-geometric setting: that is, that the Gromov-Witten invariants of $X$ are identical to those of $\tau_{b}(X)$.  There are some hints that this may not be true, though.

Consider the simple tree-level example of $\bar{\cM}_{0,n}(\P^{p|q},d)\equiv\bar{\cM}_{d}(\P^{p|q})$ studied at the end of Section \ref{Moduli}.  The space $\P^{p|q}$ is convex, in this simple unobstructed case, Kontsevich's formula for the fundamental class in terms of intersection theory should suffice \cite{Kontsevich:1995}:
\be{eqn: KFC1}
[\bar{\cM}_{d}(\P^{p|q})]=\left(\sum_{k}(-1)^{k}c(\cO^{k}_{\bar{\cM}})\right)\cap \mathrm{Td}(T_{\bar{\cM}})^{-1},
\ee
where $\cO_{\bar{\cM}}$ is the structure sheaf of $\bar{\cM}_{d}(\P^{p|q})$ with its natural $\mathbb{Z}$-grading; $c(\cO^{k}_{\bar{\cM}})$ is the homological Chern class of the appropriate sheaf, and $\mathrm{Td}(T_{\bar{\cM}})$ is the Todd class of the moduli stack.  

Furthermore, since $\P^{p|q}$ is convex, we can assume that there are no contributing factors to $\cO^{k}_{\bar{\cM}}$ coming from an obstruction theory.  By Theorem \ref{Main} and Lemma \ref{Main5}, it follows that $\cO^{0}_{\bar{\cM}}=\wedge^{\bullet}\mathcal{Q}$, where $\mathcal{Q}=(\rho_{*}(\Phi^{*}\cO(1)^{\oplus q}))^{\vee}$ via the universal instanton \eqref{eqn: univinst}.  Hence, Kontsevich's formula \eqref{eqn: KFC1} reads:
\be{eqn: KFC2}
[\bar{\cM}_{d}(\P^{p|q})]=\left(c(\cO_{\tau_{b}(\bar{\cM})})+c\left(\bigoplus_{k=1}^{\infty}\wedge^{k}(\rho_{*}(\Phi^{*}\cO(1)^{\oplus q}))^{\vee}\right)\right)\cap \mathrm{Td}(T_{\bar{\cM}})^{-1}.
\ee
The dependence of this formula on the super-structure of the stack $\bar{\cM}_{d}(\P^{p|q})$ is immediately obvious, and there is additional dependence hidden in $\mathrm{Td}(T_{\bar{\cM}})$.  Since this is the cycle in $H^{*}(\bar{\cM}_{d}(\P^{p|q}), \Q)$ which is to be integrated over to determine any Gromov-Witten invariants of the form $\la I_{0,n,d[\ell]}\ra$, it seems natural to expect these invariants to be different from those of $\P^{p}$.

The ordinary cohomology of a globally split scheme $X$ is known to be equivalent to $\tau_{b}(X)$ \cite{Bartocci:1987, Rabin:1987}; nonetheless it is possible that the Gromov-Witten theory is amenable to the super-structure, as the proposed definition of the virtual fundamental class takes the fermionic part into account.  In other words, the \emph{quantum} cohomology of $X$ is different than the quantum cohomology of $\tau_{b}(X)$.  It is also possible that one should work with a new cohomology theory for supermanifolds based on cyclic cohomology \cite{Connes:1985} or some modification thereof \cite{Varsaie:2002}.

Of course, this line of argument is quite hand-wavy.  To actually prove that Gromov-Witten theory is sensitive to the super-structure of the target manifold requires a calculation of the actual invariants (and a comparison against their counterparts in the bosonic truncation).

\subsection{Existence of Coarse Moduli Spaces}

Where moduli stacks arise in physics, one hopes to work with some bosonic scheme or algebraic space which represents the stack; if the stack is smooth and Deligne-Mumford this is equivalent to working with a space with orbifold singularities instead of the full stack.  Naturally, one can ask when such objects exist in the context of super-geometry.  In moduli problems, it is impossible to find a scheme or algebraic space that fully represents the moduli stack (i.e., a fine moduli space) as soon as there exist objects in the stack with non-trivial automorphisms.  In such cases, one instead looks for the ``coarse moduli space'' of the stack, which is defined as:
\begin{defn}[Coarse moduli space]\label{cms}
Let $X$ be an algebraic space; it is called the \emph{coarse moduli space} for a stack $\mathcal{X}\rightarrow\mathsf{Sch}$ if there is a morphism $\mathcal{X}\xrightarrow{\tau}X$, which is universal.  That is, for all algebraic spaces $Y$, there is a unique completion which makes the following diagram commutative:
\begin{equation*}
\xymatrix{
\mathcal{X} \ar[r]^{\tau} \ar[dr] & X \ar@{-->}[d]^{\exists !} \\
 & Y
 }
\end{equation*}
\end{defn}  

In the theory of bosonic stacks, the essential result for establishing the existence of a coarse moduli space (i.e., a bosonic moduli scheme) is the Keel-Mori theorem \cite{Keel:1997}, which asserts that every Deligne-Mumford bosonic stack has a coarse moduli space under very general conditions. The authors do not know whether the Keel-Mori theorem holds for super-stacks.  Although a super-geometric Keel-Mori theorem may certainly exist, a negative result would show yet another interesting facet of super-geometry whereby one is forced to work directly with the stack in moduli problems.

\subsection{Applications in Physics}

In the Introduction, we emphasized the importance of $\bar{\cM}_{g,n}(X,\beta)$ in recent advances in the understanding of planar $\cN=4$ super-Yang-Mills via twistor-string theory and related constructs.  As one might expect, twistor-string theory also contains gravitational vertex operators; unfortunately, these correspond to $\cN=4$ conformal super-gravity: a theory widely believed to be non-physical \cite{Berkovits:2004jj}.  Conformal super-gravity scattering amplitudes in flat space can be calculated directly in the string theory, and again an integral over the moduli space $\bar{\cM}_{g,n}(\P^{3|4},d)$ is required \cite{Nair:2007md, Dolan:2008gc}.  

However, a recent observation by Maldacena \cite{Maldacena:2011mk} relates tree-level graviton scattering in Einstein and conformal gravity in the presence of a cosmological constant; this means that correct (tree-level) Einstein gravity scattering amplitudes can be obtained for $\cN=0$ and $\cN=4$ from twistor-string theory \cite{Adamo:2012nn}.  Extending these ideas to maximally supersymmetric (i.e., $\cN=8$) Einstein super-gravity will certainly continue to require an integral over the moduli space, but perhaps with a non-Calabi-Yau target.  Furthermore,  deducing simplifying structures about the gravitational theory (such as a MHV-like formalism) from any twistor-string theory could follow from the properties of the moduli space, in analogy with gauge theory (c.f., \cite{Gukov:2004ei}).  In any case, we expect $\bar{\cM}_{g,n}(X,\beta)$ to play an important role in any twistorial developments in gravity as well as gauge theory.

On a more general level, the bosonic version of the moduli stack $\bar{\cM}_{g,n}(X,\beta)$ has been of much interest in the study of mirror symmetry (c.f., \cite{Cox:1999, Vafa:2003}).  The analogue of the Calabi-Yau condition for a projective scheme or variety $X$ in super-geometry is that its Berezinian sheaf has a canonical global section: $\Ber_{X}\cong\cO_{X}$ \cite{Manin:1997}.  When $X=\P^{p|q}$ this occurs when $p=q-1$ \cite{Sethi:1994ch, Schwarz:1995ak, Witten:2003nn}, and for a smooth hypersurface $V^{[s]}\subset\P^{p|q}$ of degree $s$ the condition is $p+1-q=s$ (this condition can also be generalized to hypersurfaces in weighted projective varieties) \cite{Garavuso:2011nz}.  Although there are counter-examples to Yau's theorem in super-geometry \cite{Rocek:2004bi}, they appear to be confined to fermionic dimension one \cite{Zhou:2004su, Rocek:2004ha}, so we can usually apply our intuition from ordinary differential geometry to these objects.

The notion of mirror symmetry for Calabi-Yau supermanifolds has generated interest as a candidate for incorporating rigid Calabi-Yau manifolds into the mirror symmetry landscape \cite{Sethi:1994ch, Schwarz:1995ak}.  However, most studies of mirror symmetry in super-geometry have utilized Landau-Ginzburg models to construct the mirror manifold \cite{Aganagic:2004yh, Belhaj:2004ts, Garavuso:2011nz}.  It would be very interesting to know if mirror symmetry for supermanifolds could be studied formally using $\bar{\cM}_{g,n}(X,\beta)$ or its Gromov-Witten theory, as it has for bosonic manifolds.

%%%%%%%%%%%%%%%%%%%%%%%%%%%%%%%%%%%%%%%%%%%%%%%%%%%%%%%%%%%%%%%%%%%%%%%%

\subsubsection*{\textit{Acknowledgments}}

It is a pleasure to thank Jeff Rabin and Martin Wolf for many interesting comments and suggestions.  We also thank Lionel Mason, Martin Rocek, and Dave Skinner for useful discussions and for looking at preliminary versions of this work.  TA is supported by a National Science Foundation (USA) Graduate Research Fellowship and by Balliol College; MG is supported by an EPSRC grant received under the contract EP/G027110/1.

\bibliography{sm}
\bibliographystyle{JHEP}
\end{document}